\documentclass[11pt]{amsart}
\usepackage[margin=1in]{geometry}
\usepackage{latexsym}
\usepackage{amsfonts}
\usepackage{amsmath}
\usepackage{amssymb}
\usepackage{amsthm}
\usepackage{enumerate}
\usepackage[hang,flushmargin]{footmisc}
\usepackage{caption}
\usepackage{tabu}
\usepackage{mathrsfs}
\usepackage[foot]{amsaddr}

\usepackage{graphicx}
\usepackage{caption}
 
\usepackage{bm} 
\usepackage{cite}

\makeatletter
\newtheorem*{rep@theorem}{\rep@title}
\newcommand{\newreptheorem}[2]{%
\newenvironment{rep#1}[1]{%
 \def\rep@title{#2 \ref{##1}}%
 \begin{rep@theorem}}%
 {\end{rep@theorem}}}
\makeatother

\newtheorem{theorem}{Theorem}[section]
\newreptheorem{theorem}{Conjecture}
\newtheorem{lemma}[theorem]{Lemma}
\newtheorem{proposition}[theorem]{Proposition}

\newtheorem{conjecture}[theorem]{Conjecture}
\newtheorem{question}[theorem]{Question}
\newtheorem{problem}[theorem]{Problem}

\theoremstyle{definition}
\newtheorem{definition}[theorem]{Definition}
\newtheorem{remark}[theorem]{Remark}

\DeclareMathOperator{\rc}{rc}
\DeclareMathOperator{\Des}{Des}
\DeclareMathOperator{\Comp}{Comp}
\DeclareMathOperator{\Part}{Part}
\DeclareMathOperator{\lmax}{lmax}
\DeclareMathOperator{\tl}{tl}
\DeclareMathOperator{\alt}{alt}
\DeclareMathOperator{\tla}{tla}
\DeclareMathOperator{\nib}{nib}

\begin{document}

\title[]{Quantifying Noninvertibility in Discrete Dynamical Systems} \keywords{}
\subjclass[2010]{}

\author[]{Colin Defant}
\address[]{Fine Hall, 304 Washington Rd., Princeton, NJ 08544}
\email{cdefant@princeton.edu}
\author[]{James Propp}
\address[]{74 Gilbert Road, Belmont, MA 02478}
\email{jamespropp@gmail.com}

\begin{abstract}
Given a finite set $X$ and a function $f:X\to X$, we define the \emph{degree of noninvertibility} of $f$ to be $\displaystyle\deg(f)=\frac{1}{|X|}\sum_{x\in X}|f^{-1}(f(x))|$. This is a natural measure of how far the function $f$ is from being bijective. We compute the degrees of noninvertibility of some specific discrete dynamical systems, including the Carolina solitaire map, iterates of the bubble sort map acting on permutations, bubble sort acting on multiset permutations, and a map that we call ``nibble sort." We also obtain estimates for the degrees of noninvertibility of West's stack-sorting map and the Bulgarian solitaire map. 
We then turn our attention to arbitrary functions and their iterates. In order to compare the degree of noninvertibility of an arbitrary function $f:X\to X$ with that of its iterate $f^k$, we prove that \[\max_{\substack{f:X\to X\\ |X|=n}}\frac{\deg(f^k)}{\deg(f)^\gamma}=\Theta(n^{1-1/2^{k-1}})\] for every real number $\gamma\geq 2-1/2^{k-1}$. We end with several conjectures and open problems.  
\end{abstract}

\maketitle

\section{Introduction}

Functions between finite sets play a fundamental role in classical enumerative and algebraic combinatorics, as they are often used to transfer combinatorial information from one set of objects to another. The goal of the recently-introduced field of dynamical algebraic combinatorics is to study the dynamics of functions $f:X\to X$, where $X$ is a finite set (usually of combinatorial interest). In this case, we say $f$ is a \emph{discrete dynamical system}. In both of these situations, it is very common to consider functions that are injective or even bijective. However, the world is not always so simple; some of the most intriguing combinatorial functions are not $1$-to-$1$. The purpose of this article is to introduce and explore a natural way of measuring how far a map is from being injective. We state the definition for arbitrary maps between finite sets, but we will focus our attention in most of the article on discrete dynamical systems. In this case, we are actually measuring how far the function is from being bijective.  

\begin{definition}
Let $X$ and $Y$ be finite sets, and let $f:X\to Y$ be a function. The \emph{degree of noninjectivity} of $f$ is the quantity \[\deg(f)=\frac{1}{|X|}\sum_{x\in X}|f^{-1}(f(x))|.\] When $X=Y$, we call $\deg(f)$ the \emph{degree of noninvertibility} of $f$. We often write \[\deg(f : X \to Y)\] when we wish to explicitly specify the domain $X$ and codomain $Y$.
\end{definition}

It is straightforward to verify that \[\deg(f)=\frac{1}{|X|}\sum_{y\in Y}|f^{-1}(y)|^2;\] this is the formulation that we will use in the remainder of the article. For brevity, we will often simply call $\deg(f)$ the ``degree" of $f$. For a function $f:X\to Y$, we necessarily have $1\leq\deg (f)\leq |X|$. The lower bound is attained if and only if $f$ is injective, and the upper bound is attained if and only if $f$ is a constant function. More generally, the degree of a $k$-to-$1$ map is precisely $k$. Notice that $|X|\deg(f)$ is the number of pairs $(x,x')\in X\times X$ such that $f(x)=f(x')$. Equivalently, if $x$ is chosen randomly from the uniform distribution on $X$, then $\deg(f)$ is the expected number of elements $x'\in X$ such that $f(x')=f(x)$. If we define a probability distribution $\nu$ on $X$ by $\nu(x)=|f^{-1}(x)|/|X|$, then $\log(|X|/\deg(f))$ is the R\'enyi entropy of $\nu$ of order $2$. The quantity $\deg(f)-1$ was also termed the ``coefficient of coalescence" of $f$ in \cite{Arney}.

In Section \ref{Sec:Specific}, we analyze the degrees of some specific families of discrete dynamical systems of combinatorial interest. The first types of systems are ``sorting maps" that act on permutations of length $n$; these are the bubble sort map, the (West) stack-sorting map, and a map that we call ``nibble sort" (which swaps the first pair of adjacent entries in a nonidentity permutation that are in decreasing order). We also consider analogues of bubble sort and nibble sort acting on words. Another type of system we consider is defined via chip-firing on a cycle graph and can be interpreted as a dynamical system on the set of binary strings of length $n$. The last two types of systems act on partitions of a fixed positive integer $n$; these are the Bulgarian solitaire map and its close relative, the Carolina solitaire map. 

Section \ref{Sec:Iterates} focuses on comparing $\deg(f)$ with $\deg(f^k)$, where $f^k$ denotes the $k^\text{th}$ iterate of $f$. The original motivation for doing this comes from a desire to show that the submultiplicativity condition $\deg(f\circ g)\leq\deg(f)\deg(g)$ can fail horribly, even in the specific case in which $f=g$. For an explicit small example of this failure, consider the map $f:\{1,2,3\}\to\{1,2,3\}$ given by $f(1)=2$ and $f(2)=f(3)=3$, which satisfies $\deg(f)=5/3$ and $\deg(f^2)=3$. We will prove that for every fixed positive integer $k$ and real number $\gamma\geq 2-1/2^{k-1}$, we have \[\max_{\substack{f:X\to X\\ |X|=n}}\frac{\deg(f^k)}{\deg(f)^\gamma}=\Theta(n^{1-1/2^{k-1}}).\] Setting $k=\gamma=2$ shows that the ratio $\deg(f^2)/\deg(f)^2$ can be as big as $\Theta(\sqrt{n})$ (but no bigger).  

We collect several open problems and conjectures in Section \ref{Sec:Conclusion}.

Before we proceed, let us mention the following more refined estimates for the degree, which will be useful in the remainder of the article. 

\begin{lemma}\label{Lem1} 
If $X$ and $Y$ are finite sets and $f:X\to Y$ is a function, then \[\frac{|X|}{|f(X)|}\leq\deg(f)\leq\max_{y\in Y}|f^{-1}(y)|.\]
\end{lemma}
\begin{proof}
The Cauchy-Schwarz inequality tells us that \[\deg(f)=\frac{1}{|X|}\sum_{y\in f(X)}|f^{-1}(y)|^2\geq\frac{1}{|X|}\frac{1}{|f(X)|}\left(\sum_{y\in f(X)}|f^{-1}(y)|\right)^2=\frac{1}{|X|\,|f(X)|}|X|^2=\frac{|X|}{|f(X)|}.\] For the upper bound, let $\displaystyle M=\max_{y\in Y}|f^{-1}(y)|$. We have \[\deg(f)=\frac{1}{|X|}\sum_{y\in Y}|f^{-1}(y)|^2\leq\frac{1}{|X|}\sum_{y\in Y}|f^{-1}(y)|\cdot M=M. \qedhere\]
\end{proof}

We also pause to relate our work to a natural equivalence relation on endofunctions of a finite set. Given sets $X,Y,Z$ of the same cardinality, we say that two functions $f:X\to Y$ and $g:Y\to Z$ are \emph{pseudoconjugate} if there exist bijections $h:Y\to Z$ and $\widetilde h:X\to Y$ such that $h\circ f=g\circ \widetilde h$. This notion will appear in Sections~\ref{Sec:Chip}, where we will see that two combinatorially-defined maps $\nib:\{0,1\}^n\to\{0,1\}^n$ and $\chi:\{0,1\}^n\to\{0,1\}^n$ are pseudoconjugate even though they are not actually conjugate. 

\begin{proposition}\label{Prop:Pseudoconjugate}
Let $X,Y,Z$ be sets of the same cardinality. Two functions $f:X\to Y$ and $g:Y\to Z$ are pseudoconjugates of each other if and only if there exists a bijection $h:Y\to Z$ such that $|f^{-1}(y)|=|g^{-1}(h(y))|$ for all $y\in Y$. 
\end{proposition}

\begin{proof}
First suppose $f$ and $g$ are pseudoconjugate, and let $h:Y\to Z$ and $\widetilde h:X\to Y$ be bijections such that $h\circ f=g\circ h'$. For every $y\in Y$, we have $|f^{-1}(y)|=|\widetilde h^{-1}(g^{-1}(h(y)))|=|g^{-1}(h(y))|$.
 
To prove the converse, suppose there exists a bijection $h:Y\to Z$ such that $|f^{-1}(y)|=|g^{-1}(h(y))|$ for all $y\in Y$. For each $y\in Y$, there is a bijection $\alpha_y:f^{-1}(y)\to g^{-1}(h(y))$. Now simply define $\widetilde h:X\to Y$ by $\widetilde h(x)=\alpha_{f(x)}(x)$ for all $x\in X$. We readily check that $g(\widetilde h(x))=g(\alpha_{f(x)}(x))=h(f(x))$ for every $x\in X$, so $f$ and $g$ are pseudoconjugate. 
\end{proof}

It is immediate from the previous proposition that two pseudoconjugate maps must have the same degree.

\section{Specific Discrete Dynamical Systems}\label{Sec:Specific}

In this section, we consider specific families of discrete dynamical systems that are indexed by positive integers. In order to define these maps, we recall the following standard definitions. 

A \emph{permutation} is an ordering of a finite set of positive integers. We view permutations as words. Let $S_n$ be the set of permutations of the set $[n]:=\{1,\ldots,n\}$. The \emph{normalization} of a permutation $\pi=\pi_1\cdots\pi_n$ is the permutation in $S_n$ obtained by replacing the $i^\text{th}$-smallest entry in $\pi$ with $i$ for all $i\in[n]$. A \emph{descent} of a permutation $\pi=\pi_1\cdots\pi_n$ is an index $i\in[n-1]$ such that $\pi_i>\pi_{i+1}$. Let $\Des(\pi)$ be the set of descents of $\pi$. Given a tuple of positive integers ${\bf a}=(a_1,\ldots,a_r)$, we define $\mathcal W_{{\bf a}}$ to be the set of all words over the alphabet $[r]$ that contain exactly $a_i$ copies of the letter $i$ for all $i$. A \emph{composition} of a positive integer $n$ is a tuple $c=(c_1,\ldots,c_\ell)$ of positive integers that sum to $n$. The entries $c_1,\ldots,c_\ell$ are called the \emph{parts} of $c$. A \emph{partition} of a positive integer $n$ is a composition of $n$ whose parts appear in nonincreasing order. Let $\Comp(n)$ and $\Part(n)$ denote, respectively, the set of compositions of $n$ and the set of partitions of $n$. 

\subsection{Bubble Sort and Its Iterates}\label{Sec:Bubble}

Suppose $\pi$ is a permutation of length $n$ and $i\in[n-1]$. If $i\in \Des(\pi)$, let $t_i(\pi)$ be the permutation obtained from $\pi$ by swapping the $i^\text{th}$ and $(i+1)^{\text{st}}$ entries in $\pi$. If $i\not\in\Des(\pi)$, let $t_i(\pi)=\pi$. The operators $t_1,\ldots,t_{n-1}$ play an important role in algebraic combinatorics because they generate the $0$-Hecke algebra of the symmetric group $S_n$ \cite{Hivert}. 

Let ${\bf B}(\pi)=t_{n-1}\circ t_{n-2}\circ\cdots\circ t_1(\pi)$. The function ${\bf B}$ is called the \emph{bubble sort map}. An alternative recursive description of ${\bf B}$ is as follows. First, ${\bf B}$ sends the empty permutation to itself. If $\pi$ is a permutation with largest entry $m$, then we can write $\pi=LmR$ for some permutations $L$ and $R$. Then ${\bf B}(\pi)={\bf B}(L)Rm$. For example, \[{\bf B}(416352)={\bf B}(41)\,3526=143526.\] By a slight abuse of notation, we will use the same letter ${\bf B}$ to denote the restriction of ${\bf B}$ to $S_n$, which we can view as a discrete dynamical system on $S_n$. We refer the interested reader to \cite{Chung} and \cite[pages 106--110]{Knuth2} for additional information and interesting properties of the bubble sort map. 
For instance, Knuth reports that for $k\leq n-1$, the number of permutations of length $n$ that are completely sorted after $k$ iterations of ${\bf B}$ is $(k+1)^{n-k-1}(k+1)!$; from our point of view, this is a formula for the cardinality of the $k$-fold preimage of the identity permutation.

For every fixed positive integer $k$, we will give an exact formula for the degree of the iterate ${\bf B}^k:S_n\to S_n$. We first need the following preliminary definitions. Given a permutation $\pi=\pi_1\cdots\pi_n\in S_n$, let $e_j(\pi)$ be the number of entries of $\pi$ that appear to the left of $j$ and are bigger than $j$. The tuple $(e_1(\pi),\ldots,e_n(\pi))$ is called the \emph{inversion table} of $\pi$. For example, the inversion tables of $416352$ and $143526$ are $(1,4,2,0,1,0)$ and $(0,3,1,0,0,0)$, respectively. Let $I_n$ denote the set of tuples $(e_1,\ldots,e_n)$ such that $0\leq e_i\leq n-i$ for all $i$. It is not difficult to show that the map $\pi\mapsto (e_1(\pi),\ldots,e_n(\pi))$ is a bijection from $S_n$ to $I_n$. A \emph{left-to-right maximum} of $\pi$ is an entry $j$ such that $e_j(\pi)=0$. Let $\lmax(\pi)$ denote the number of left-to-right maxima of $\pi$. 
The \emph{tail length} of $\pi$, denoted $\tl(\pi)$, is the largest integer $\ell\in\{0,\ldots,n\}$ such that $\pi_i=i$ for all $i\in\{n-\ell+1,\ldots,n\}$. For example, $\tl(23145)=2$, $\tl(23154)=0$, and $\tl(12345)=5$.  

\begin{lemma}\label{Lem2}
Let $k$ be a nonnegative integer. For every permutation $\pi$, we have \[|{\bf B}^{-k}(\pi)|=\begin{cases} 0, & \mbox{if } \tl(\pi)<k; \\ k!(k+1)^{\lmax(\pi)-k}, & \mbox{if } \tl(\pi)\geq k. \end{cases}\]
\end{lemma}

\begin{proof}
It is known (and straightforward to prove using the recursive description mentioned above) that bubble sort has the effect of decreasing each nonzero entry in the inversion table of a permutation by $1$. In other words, $e_j({\bf B}(\sigma))=\max\{0,e_j(\sigma)-1\}$ for all $j$. This implies that $e_\ell({\bf B}^k(\sigma))=0$ for each $\ell\in\{n-k+1,\ldots,n\}$ and $e_\ell({\bf B}^k(\sigma))\leq n-k-\ell$ for each $\ell\in\{1,\ldots,n-k\}$. These two conditions are equivalent to the statement that $\tl({\bf B}^k(\sigma))\geq k$ (indeed, the second condition says that $(e_1(\sigma),\ldots,e_{n-k}(\sigma))$ is the inversion table of a permutation in $S_{n-k}$). Thus, permutations with tail lengths less than $k$ have $0$ preimages under ${\bf B}^k$. Now suppose $\tl(\pi)\geq k$. Choosing $\sigma\in{\bf B}^{-k}(\pi)$ is equivalent to choosing the inversion sequence of $\sigma$. To do this, we must first increase each of the nonzero entries in $(e_1(\pi),\ldots,e_{n-k}(\pi))$ by $k$. Next, we must increase each of the $\lmax(\pi)-k$ entries in $(e_1(\pi),\ldots,e_{n-k}(\pi))$ that are equal to $0$ by some integer in $\{0,\ldots,k\}$. There are $(k+1)^{\lmax(\pi)-k}$ ways to do this. Finally, for each $\ell\in\{n-k+1,\ldots,n\}$, we must increase the entry $e_\ell(\pi)$ (which is $0$) by some integer in $\{0,\ldots,n-\ell\}$. The total number of ways to do this for all $\ell\in\{n-k+1,\ldots,n\}$ is $k!$.   
\end{proof}

\begin{theorem}\label{Thm1}
For every positive integer $k$, we have \[\deg({\bf B}^k:S_n\to S_n)=\frac{(n+k^2+k)!(k!)^2}{n!(k^2+2k)!}.\]
\end{theorem}

\begin{proof}
A well-known result due to R\'enyi \cite{Renyi} states that the number of elements of $S_n$ with $\ell$ left-to-right maxima is given by the unsigned Stirling number of the first kind $\begin{bmatrix}n \\ \ell\end{bmatrix}$. These numbers have the generating function \[
\sum_{\ell=0}^n\begin{bmatrix}n \\ \ell\end{bmatrix}x^\ell=x(x+1)(x+2)\cdots(x+n-1).\] Given $\pi\in S_n$ with $\tl(\pi)\geq k$, we can write $\pi=\tau(n-k+1)(n-k+2)\cdots n$ for some $\tau\in S_{n-k}$ with $\lmax(\tau)=\lmax(\pi)-k$. Invoking Lemma \ref{Lem2}, we find that \[\deg({\bf B}^k:S_n\to S_n)=\frac{1}{n!}\sum_{\pi\in S_n}|{\bf B}^{-k}(\pi)|^2=\frac{1}{n!}\sum_{\substack{\pi\in S_n\\ \tl(\pi)\geq k}}(k!)^2((k+1)^2)^{\lmax(\pi)-k}\] \[=\frac{1}{n!}\sum_{\tau\in S_{n-k}}(k!)^2((k+1)^2)^{\lmax(\tau)}=\frac{(k!)^2}{n!}\sum_{\ell=0}^{n-k}\begin{bmatrix}n-k \\ \ell\end{bmatrix}((k+1)^2)^\ell\] \[=\frac{(k!)^2}{n!}(k+1)^2((k+1)^2+1)((k+1)^2+2)\cdots ((k+1)^2+n-k-1)=\frac{(n+k^2+k)!(k!)^2}{n!(k^2+2k)!}.\qedhere\]
\end{proof}

When $k=1$, Theorem \ref{Thm1} tells us that the degree of ${\bf B}:S_n\to S_n$ is $\dfrac{(n+1)(n+2)}{6}$. There is an alternative way of generalizing this result that makes use of the probabilistic interpretation of the degree mentioned in the introduction. Namely, the bubble sort map is associated with a random variable $U_n$ on $S_n$ defined by $U_n(\pi)=|{\bf B}^{-1}({\bf B(\pi)})|$. The degree $\deg({\bf B}:S_n\to S_n)$ is just the expected value of $U_n$. The following theorem computes all of the moments of this random variable.

\begin{theorem}
For $m\geq 1$, the $m^\text{th}$ moment of $U_n$ (with respect to the uniform distribution on $S_n$) is given by \[\frac{1}{n!}\sum_{\pi\in S_n}|{\bf B}^{-1}({\bf B(\pi)})|^m=\prod_{j=1}^{n-1}\frac{2^{m+1}+n-j-1}{n-j+1}.\]
\end{theorem}

\begin{proof}
For each $j\in[n-1]$, define a function $V_j:\{0,\ldots,n-j\}\to\mathbb R$ by $V_j(t)=2$ if $t\in\{0,1\}$ and $V_j(t)=1$ otherwise. As mentioned in the proof of Lemma \ref{Lem2}, the bubble sort map has the effect of decreasing by $1$ each of the positive entries in the inversion table of a permutation. It follows that for $\pi\in S_n$, we have $|{\bf B}^{-1}({\bf B}(\pi))|=\prod_{j=1}^{n-1}V_j(e_j(\pi))$. Choosing $\pi\in S_n$ uniformly at random is equivalent to choosing the entries $e_j=e_j(\pi)\in\{0,\ldots,n-j\}$ (for $j\in[n-1]$) independently and uniformly at random (note that $e_n(\pi)$ is always $0$). Therefore, the expected value of $U_n(\pi)^m$ is the same as the expected value of $\prod_{j=1}^{n-1}V_j(e_j)^m$. The expected value of $V_j(e_j)^m$ is \[\dfrac{2}{n-j+1}\cdot 2^m+\dfrac{n-j-1}{n-j+1}\cdot 1^m=\dfrac{2^{m+1}+n-j-1}{n-j+1},\] so the desired result follows from the fact that $e_1,\ldots,e_{n-1}$ are chosen independently.  
\end{proof}

\subsection{Bubble Sort for Words}
Throughout this subsection, fix a tuple ${\bf a}=(a_1,\ldots,a_r)$ of positive integers, where $r\geq 2$. We are interested in the obvious analogue of the bubble sort map acting on $\mathcal W_{\bf a}$. Recall that this is the set of words with exactly $a_i$ copies of the letter $i$ for all $i$ (these are also called permutations of the multiset $\{1^{a_1},\ldots,r^{a_r}\}$). Given a word $w=w_1\cdots w_\ell$ over the alphabet of positive integers and $i\in[\ell-1]$ with $w_i>w_{i+1}$, let $t_i(w)$ be the word obtained by swapping the positions of $w_i$ and $w_{i+1}$ in $w$. If $w_i\leq w_{i+1}$, let $t_i(w)=w$. Let ${\bf B}(w)=t_{\ell-1}\circ t_{\ell-2}\circ\cdots\circ t_1(w)$. We can consider this generalization of bubble sort as a discrete dynamical system on $\mathcal W_{\bf a}$.  

\begin{theorem}
The degree of ${\bf B}:\mathcal W_{\bf a}\to\mathcal W_{\bf a}$ is given by \[\deg({\bf B}:\mathcal W_{\bf a}\to\mathcal W_{\bf a})=\prod_{j=1}^{r-1}\left(2\frac{a_j}{a_{j+1}+a_{j+2}+\cdots+a_r+1}+1\right).\]
\end{theorem}

\begin{proof}
The proof is by induction on the length $r$ of the tuple ${\bf a}=(a_1,\ldots,a_r)$. Let us first assume $r=2$. Every word in $\mathcal W_{(a_1,a_2)}$ can be written in the form $1^{\gamma_0}21^{\gamma_1}2\cdots 21^{\gamma_{a_2}}$, where $(\gamma_0,\gamma_1,\ldots,\gamma_{a_2})$ is a tuple of nonnegative integers that sum to $a_1$. In fact, this establishes a bijection between $\mathcal W_{(a_1,a_2)}$ and the set of $(a_2+1)$-tuples of nonnegative integers that sum to $a_1$. 
Bubble sort transforms the word corresponding to the tuple $(\gamma_0,\gamma_1,\ldots,\gamma_{a_2})$ to the word corresponding to the tuple $(\gamma_0+\gamma_1,\gamma_2,\ldots,\gamma_{a_2},0)$. The number of preimages under ${\bf B}$ of the word corresponding to $(\gamma_0+\gamma_1,\gamma_2,\ldots,\gamma_{a_2},0)$ is $\gamma_0+\gamma_1+1$ (since this is the number of ways to write $\gamma_0+\gamma_1$ as a sum of two nonnegative integers). Now let $\overline\gamma_i$ be the average value of $\gamma_i$ over $\mathcal W_{(a_1,a_2)}$. By symmetry, $\overline\gamma_i$ is independent of $i$. We know that $\sum_{i=0}^{a_2}\overline\gamma_i=a_1$, so $\overline\gamma_i=\dfrac{a_1}{a_2+1}$. It follows that $\deg({\bf B}:\mathcal W_{(a_1,a_2)}\to\mathcal W_{(a_1,a_2)})$, which is the average value of $\gamma_1+\gamma_2+1$ over $\mathcal W_{(a_1,a_2)}$, is $2\dfrac{a_1}{a_2+1}+1$. 

We now assume ${\bf a}=(a_1,\ldots,a_r)$, where $r\geq 3$. Let ${\bf a'}=(a_1,\ldots,a_{r-2},a_{r-1}+a_r)$ and ${\bf a''}=(a_{r-1},a_r)$. There is a natural projection $\psi:\mathcal W_{\bf a}\to\mathcal W_{{\bf a}'}$ obtained by replacing each occurrence of the letter $r$ in a word with the letter $r-1$. We also have a map $\varphi:\mathcal W_{{\bf a}}\to\mathcal W_{{\bf a}''}$ that decreases each letter in a word by $r-2$ and then deletes all of the nonpositive letters. For example, with ${\bf a}=(2,1,2,3)$, we have \[\psi(14234413)=13233313\quad\text{and}\quad\varphi(14234413)=21221.\] It is straightforward to check from these definitions that $\psi$ and $\varphi$ commute with the action of bubble sort. That is, $(\psi({\bf B}(w)),\varphi({\bf B}(w)))=({\bf B}(\psi(w)),{\bf B}(\varphi(w)))$. Moreover, the map $\mathcal W_{\bf a}\to\mathcal W_{\bf a'}\times\mathcal W_{\bf a''}$ given by $w\mapsto(\psi(w),\varphi(w))$ is a bijection. This means that if $(\psi(w),\varphi(w))=(w',w'')$, then $|{\bf B}^{-1}(w)|=|{\bf B}^{-1}(w')|\cdot|{\bf B}^{-1}(w'')|$. Consequently, \[\deg({\bf B}:\mathcal W_{\bf a}\to\mathcal W_{\bf a})=\frac{1}{|\mathcal W_{\bf a}|}\sum_{w\in\mathcal W_{\bf a}}|{\bf B}^{-1}(w)|^2=\frac{1}{|\mathcal W_{\bf a'}|}\frac{1}{|\mathcal W_{\bf a''}|}\sum_{(w',w'')\in\mathcal W_{\bf a'}\times\mathcal W_{\bf a''}}|{\bf B}^{-1}(w')|^2|{\bf B}^{-1}(w'')|^2\] \[=\deg({
\bf B}:\mathcal W_{{\bf a}'}\to \mathcal W_{{\bf a}'})\deg({
\bf B}:\mathcal W_{{\bf a}''}\to \mathcal W_{{\bf a}''}).\] We know from the $r=2$ case that $\deg({
\bf B}:\mathcal W_{{\bf a}''}\to \mathcal W_{{\bf a}''})=2\dfrac{a_{r-1}}{a_{r}+1}+1$, so the desired result follows by induction on $r$. 
\end{proof}

\subsection{The Stack-Sorting Map}
The stack-sorting map was originally defined in West's Ph.D. dissertation \cite{West} as a deterministic variant of a ``stack-sorting algorithm" introduced in Knuth's book \emph{The Art of Computer Programming} \cite{Knuth}. It has a recursive definition very similar to the that of the bubble sort map. First, $s$ sends the empty permutation to itself. If $\pi$ is a permutation with largest entry $m$, then we can write $\pi=LmR$ for some permutations $L$ and $R$. Then $s(\pi)=s(L)s(R)m$. For example, \[s(416352)=s(41)\,s(352)\,6=s(1)\,4\,s(3)\,s(2)\,56=143256.\] For each positive integer $n$, we can view $s$ as a discrete dynamical system on $S_n$. We refer the reader to \cite{Bona, BonaSurvey, DefantCounting, DefantPreimages} and the references therein for more information about this map. 

In \cite{DefantCounting, DefantPostorder, DefantPreimages}, the first author found methods for computing the number of preimages of an arbitrary permutation under the stack-sorting map. Unfortunately, it seems quite difficult to use these methods in order to find an explicit formula for $\deg(s:S_n\to S_n)$. However, we will still be able to show that $\deg(s:S_n\to S_n)$ grows exponentially in $n$. This is in stark contrast to Theorem \ref{Thm1}, which shows that for each fixed $k$, the degree of ${\bf B}^k:S_n\to S_n$ grows polynomially in $n$. Roughly speaking, this says that the stack-sorting map is much further from being invertible than any iterate of the bubble sort map. 

\begin{theorem}\label{Thm2}
The limit $\lim\limits_{n\to \infty}\deg(s:S_n\to S_n)^{1/n}$ exists and satisfies \[1.12462\leq\lim_{n\to\infty}\deg(s:S_n\to S_n)^{1/n}\leq 4.\]
\end{theorem}
\begin{proof}
For convenience, let $d_n=\deg(s:S_n\to S_n)$. Note that $n!d_n$ is the number of pairs $(\pi,\pi')\in S_n\times S_n$ such that $s(\pi)=s(\pi')$. Suppose $\pi,\pi'\in S_{m-1}$ and $\sigma,\sigma'\in S_{n-1}$ are such that $s(\pi)=s(\pi')$ and $s(\sigma)=s(\sigma')$. Let $A$ be an $(m-1)$-element subset of $\{1,\ldots,m+n-2\}$. Let $\widetilde\pi$ and $\widetilde\pi'$ be the permutations of $A$ whose normalizations are $\pi$ and $\pi'$, respectively. Let $\widetilde\sigma$ and $\widetilde\sigma'$ be the permutations of $\{1,\ldots,m+n-2\}\setminus A$ whose normalizations are $\sigma$ and $\sigma'$, respectively. We have $s(\widetilde\pi)=s(\widetilde\pi')$ and $s(\widetilde\sigma)=s(\widetilde\sigma')$, so \[s(\widetilde\pi(m+n-1)\widetilde\sigma)=s(\widetilde\pi)s(\widetilde\sigma)(m+n-1)=s(\widetilde\pi')s(\widetilde\sigma')(m+n-1)=s(\widetilde\pi'(m+n-1)\widetilde\sigma').\] This shows that $\widetilde\pi(m+n-1)\widetilde\sigma$ and $\widetilde\pi'(m+n-1)\widetilde\sigma'$ are two elements of $S_{m+n-1}$ with the same image under $s$. The map sending the tuple $(\pi,\pi',\sigma,\sigma',A)$ to the pair $(\widetilde\pi(m+n-1)\widetilde\sigma,\widetilde\pi'(m+n-1)\widetilde\sigma')$ is injective, so \[((m-1)!d_{m-1})((n-1)!d_{n-1}){m+n-2\choose m-1}\leq (m+n-1)!d_{m+n-1}.\] Rearranging, this shows that 
\begin{equation}\label{Eq1}
d_{m-1}d_{n-1}\leq(m+n-1)d_{m+n-1}.
\end{equation} 

We will make use of a generalization of Fekete's lemma due to de Bruijn and Erd\H{o}s \cite{deBruijn}, which states that if a sequence of positive real numbers $(a_m)_{m\geq 1}$ satisfies $a_ma_n\leq a_{m+n}$ whenever $1/2\leq n/m\leq 2$, then $\lim\limits_{n\to\infty}a_n^{1/n}$ exists and equals $\sup\limits_{n\geq 1}a_n^{1/n}$. Now let $a_n=\dfrac{d_{n-1}}{n^2}$ for $n\geq 8$ and $a_n=0$ for $1\leq n\leq 7$. It is not difficult to check that $m^2n^2\geq (m+n)^2(m+n-1)$ whenever $m,n\geq 8$ and $1/2\leq n/m\leq 2$. Therefore, it follows from \eqref{Eq1} that \[a_ma_n=\frac{d_{m-1}d_{n-1}}{m^2n^2}\leq\frac{(m+n-1)d_{m+n-1}}{(m+n)^2(m+n-1)}=a_{m+n}\] whenever $m,n\geq 8$ and $1/2\leq n/m\leq 2$. The inequality $a_ma_n\leq a_{m+n}$ also certainly holds whenever $m$ or $n$ is at most $7$. According to the aforementioned generalization of Fekete's lemma, $\lim\limits_{n\to\infty}a_n^{1/n}$ exists and equals $\sup\limits_{n\geq 1}a_n^{1/n}$. It now follows from the definition of $a_n$ that $\lim\limits_{n\to\infty}d_n^{1/n}$ exists and equals $\sup\limits_{n\geq 1}a_n^{1/n}$. We can use a computer to verify that $a_{10}^{1/10}\geq 1.12462$, so this proves the lower bound in the theorem.  

To prove the upper bound, we use the fact that $|s^{-1}(\pi)|\leq C_n$ for all $\pi\in S_n$, where $\displaystyle C_n=\frac{1}{n+1}{2n\choose n}$ is the $n^\text{th}$ Catalan number (see Exercise 23 in \cite[Chapter 8]{Bona}). It follows from Lemma~\ref{Lem1} that $\deg(s:S_n\to S_n)\leq C_n$, so \[\lim\limits_{n\to\infty}\deg(s:S_n\to S_n)^{1/n}\leq\lim\limits_{n\to\infty}C_n^{1/n}=4. \qedhere\]  
\end{proof}

\subsection{Nibble Sort}

Recall the definition of the maps $t_i:S_n\to S_n$ from Section \ref{Sec:Bubble}. In this section, we consider the nibble sort map $\nib:S_n\to S_n$ defined by $\nib(\pi)=t_{\min(\Des(\pi))}(\pi)$ if $\pi\neq 123\cdots n$ and $\nib(123\cdots n)=123\cdots n$. In other words, if $\pi\in S_n\setminus\{123\cdots n\}$ has initial descent $i$, then $\nib(\pi)$ is the permutation obtained from $\pi$ by swapping the entries $\pi_i$ and $\pi_{i+1}$. 

In the previous two subsections, we found that $\deg({\bf B}:S_n\to S_n)$ grows quadratically in $n$ while $\deg(s:S_n\to S_n)$ grows exponentially in $n$. Since the map $\nib$ does not change its input very much (it just nibbles a little bit), one might expect $\deg(\nib:S_n\to S_n)$ to grow much slower than $\deg({\bf B}:S_n\to S_n)$. In fact, $\deg(\nib:S_n\to S_n)$ approaches a constant as $n\to\infty$. 

\begin{theorem}\label{Thm4}
For every $n\geq 1$, we have \[\deg(\nib:S_n\to S_n)=\frac{(n-1)(n-2)^2+n^2}{n!}+\sum_{k=1}^{n-2}\frac{k(k^3-k+1)}{(k+2)!}.\] Furthermore, \[\lim_{n\to\infty}\deg(\nib:S_n\to S_n)=4e-9\approx 1.87313.\]
\end{theorem}

\begin{proof}
We need to compute $|\nib^{-1}(\pi)|$ for each $\pi\in S_n$. For this, we consider some cases. The case in which $\pi=123\cdots n$ is simple since \[|\nib^{-1}(123\cdots n)|=|\{123\cdots n\}\cup\{123\cdots (i-1)(i+1)i(i+2)\cdots n:1\leq i\leq n-1\}|=n.\] If $\Des(\pi)=\{n-1\}$, then $\nib^{-1}(\pi)$ is the set of permutations obtained from $\pi$ by swapping the entries $\pi_i$ and $\pi_{i+1}$ for some $i\in[n-2]$. There are precisely $n-1$ permutations $\pi$ such that $\Des(\pi)=\{n-1\}$, and we have seen that each such permutation has $n-2$ preimages under $\nib$. In total, the cases $\pi=123\cdots n$ and $\Des(\pi)=\{n-1\}$ contribute the term $\dfrac{(n-1)(n-2)^2+n^2}{n!}$ to $\deg(\nib:S_n\to S_n)$.  

Now fix $k\leq n-2$. We are going to count permutations $\pi\in S_n$ with $\min(\Des(\pi))=k$ and $\pi_k<\pi_{k+2}$. This is equivalent to counting permutations whose first $k+2$ entries have normalization $123\cdots(j-1)(j+1)\cdots(k+1)j(k+2)$ for some $j \in [k]$. For each fixed $j$, the probability that the first $k+2$ entries of a permutation chosen uniformly at random from $S_n$ have this normalization is $\dfrac{1}{(k+2)!}$. 
Since there are $k$ choices for $j$, the probability that the first $k+2$ entries of a random permutation have one of these normalizations is $\dfrac{k}{(k+2)!}$. Therefore, the number of such permutations is $\dfrac{n!k}{(k+2)!}$. For each such permutation $\pi$, $\nib^{-1}(\pi)$ is the set of permutations obtained from $\pi$ by swapping the entries $\pi_i$ and $\pi_{i+1}$ for some $i\in[k-1]\cup\{k+1\}$; thus, $|\nib^{-1}(\pi)|=k$. This case gives a contribution of $\dfrac{k^3}{(k+2)!}$ to $\deg(\nib:S_n\to S_n)$. 

A probabilistic argument similar to the one used in the previous paragraph shows that the number of permutations in $S_n$ with smallest descent $k$ is $\dfrac{n!k}{(k+1)!}$ (this is also known: see OEIS sequence A092582 \cite{OEIS}). Therefore, the number of permutations $\pi\in S_n$ satisfying $\min(\Des(\pi))=k$ and $\pi_k>\pi_{k+2}$ is $\dfrac{n!k}{(k+1)!}-\dfrac{n!k}{(k+2)!}=\dfrac{n!k(k+1)}{(k+2)!}$. For each such $\pi$, $\nib^{-1}(\pi)$ is the set of permutations obtained from $\pi$ by swapping the entries $\pi_i$ and $\pi_{i+1}$ for some $i\in[k-1]$; thus, $|\nib^{-1}(\pi)|=k-1$. This final case contributes $\dfrac{(k-1)^2k(k+1)}{(k+2)!}$ to $\deg(\nib:S_n\to S_n)$. 

Combining these contributions yields \[\deg(\nib:S_n\to S_n)=\frac{(n-1)(n-2)^2+n^2}{n!}+\sum_{k=1}^{n-2}\left(\frac{k^3}{(k+2)!}+\frac{(k-1)^2k(k+1)}{(k+2)!}\right)\] \[=\frac{(n-1)(n-2)^2+n^2}{n!}+\sum_{k=1}^{n-2}\frac{k(k^3-k+1)}{(k+2)!}.\] Consequently, \[\lim_{n\to\infty}\deg(\nib:S_n\to S_n)=\sum_{k=1}^{\infty}\frac{k(k^3-k+1)}{(k+2)!}.\] Using Mathematica (or by explicitly computing exponential generating functions of polynomial sequences), one can check that $\displaystyle\sum_{k=1}^{\infty}\frac{k(k^3-k+1)}{(k+2)!}=4e-9$.  
\end{proof}

\subsection{Nibble Sort for Binary Words}
There is a natural analogue of the nibble sort map (which we also denote by $\nib$) that acts on binary strings. Namely, if $w\in\{0,1\}^n$, then $\nib(w)$ is obtained from $w$ by replacing the first occurrence of the factor $10$ in $w$ with $01$. If no such factor exists, then $\nib(w)=w$. For example, $\nib(001110)=\nib(010101)=001101$ and $\nib(00011)=00011$. In what follows, we let $0^\alpha$ denote the word $00\cdots 0$ consisting of $\alpha$ copies of the letter $0$. The word $1^\alpha$ is defined similarly. 

\begin{theorem}\label{Thm:BinaryNibble}
For every $n\geq 2$, the degree of $\nib:\{0,1\}^n\to\{0,1\}^n$ is $3/2$. 
\end{theorem}

\begin{proof}
Let $Z$ be the set of words in $\{0,1\}^n$ of the form $0^\gamma 11v$ for some integer $\gamma\geq 1$ and some word $v\in\{0,1\}^{n-\gamma-2}$. Let $Z'=Z\cup\{0^{n-1}1\}$. We claim that if a word $w\in\{0,1\}^n$ satisfies $|\nib^{-1}(w)|\geq 2$, then in fact $|\nib^{-1}(w)|=2$ and $|\nib^{-1}(w)\cap Z'|=1$. 

To see this, we first consider the case in which one of the elements of $\nib^{-1}(w)$ is $w$ itself. This means that $w=0^\alpha 1^\beta$ for some $\alpha,\beta\geq 0$. The assumption that $|\nib^{-1}(w)|\geq 2$ tells us that $w\not\in\{0^n,1^n\}$, so $\alpha,\beta\geq 1$. It is now straightforward to verify that $\nib^{-1}(w)=\{w,0^{\alpha-1}101^{\beta-1}\}$. Since $w$ is in $Z'$ and $0^{\alpha-1}101^{\beta-1}$ is not, the claim is satisfied in this case.

Next, assume $w\not\in\nib^{-1}(w)$. Each preimage of $w$ is obtained by changing an occurrence of the factor $01$ in $w$ to $10$. Thus, $w$ must have at least two occurrences of $01$, and must therefore have the form $w=x01y01z$ for some words $x,y,z$ such that $\nib(x10y01z)=\nib(x01y10z)=w$. Because $\nib(x01y10z)=w$, the word $x01y$ does not contain an occurrence of the factor $10$; hence, $x=0^{\alpha-1}$ and $y=1^{\beta-1}$ for some $\alpha,\beta \geq 1$. This means that $w = 0^{\alpha-1} 0 1 1^{\beta-1} 0 1 z = 0^{\alpha} 1^{\beta} 0 1 z$. Note that there is no way to obtain an element of $\nib^{-1}(w)$ from $w$ by changing an occurrence of $01$ in $z$ to $10$. It follows that $\nib^{-1}(w)=\{0^{\alpha-1}101^{\beta-1}01z,0^{\alpha}1^{\beta}10z\}$, and the claim is satisfied once again. 

We have seen that every element of $\{0,1\}^n$ has at most $2$ preimages under $\nib$, so we can write $\{0,1\}^n=A_0\cup A_1\cup A_2$, where $A_i=\{w\in\{0,1\}^n:|\nib^{-1}(w)|=i\}$. We obtain a map $A_2\to Z'$ by sending each $w\in A_2$ to the unique element of $\nib^{-1}(w)\cap Z'$. This map is clearly injective; we claim that it is also surjective. Proving this claim amounts to showing that $\nib(u)\in A_2$ for every $u\in Z'$. If $u$ is of the form $0^\gamma 1^{\delta}$ for some $\gamma,\delta\geq 1$, then $\nib(u)=\nib(0^{\gamma-1}101^{\delta-1})=u$, so $\nib(u)\in A_2$. If $u$ is not of this form, then it must be of the form $0^\gamma 1^\delta0x$ for some $\gamma\geq 1$, some $\delta\geq 2$, and some word $x$. In this case, we have $\nib(u)=0^{\gamma}1^{\delta-1}01x=\nib(0^{\gamma-1}101^{\delta-2}01x)$, so $\nib(u)\in A_2$ in this case as well. This proves the surjectivity, so we now know that $|A_2|=|Z'|$. 

Let $Y$ be the set of all words of the form $0^\gamma 10v$ for some integer $\gamma\geq 1$ and some $v\in\{0,1\}^{n-\gamma-2}$. There is an obvious bijection between $Y$ and $Z$. Furthermore, $Y\cup Z\cup\{0^n,0^{n-1}1\}$ is precisely the set of words in $\{0,1\}^n$ that begin with the letter $0$. It follows that $|Z|=2^{n-2}-1$, so $|A_2|=|Z'|=2^{n-2}$. Because $2^n=|\nib^{-1}(\{0,1\}^n)|=|A_1|+2|A_2|$, we have $|A_1|=2^{n-1}$. Finally, \[\deg(\nib:\{0,1\}^n\to\{0,1\}^n)=\frac{1}{2^n}(|A_0|\cdot 0^2+|A_1|\cdot 1^2+|A_2|\cdot 2^2)=\frac 32. \qedhere\]
\end{proof}

\subsection{Binary Chip-Firing on a Cycle}\label{Sec:Chip}

In this subsection, we study a map on binary strings defined in terms of chip-firing on a cycle graph with $n+1$ vertices. In this set-up, there are two special vertices of the cycle, called the \emph{source} and the \emph{sink}, that are adjacent to each other. Suppose we have a configuration of chips sitting on the non0sink vertices of the cycle. Whenever a non-sink vertex has at least $2$ chips on it, the vertex fires $1$ chip to each of its neighbors; whenever a chip arrives at the sink, it disappears. Firings take place until no more firings are possible; it is known that the order of firings does not affect the final configuration of chips. Notice that a configuration is stable (i.e., no firings can occur) if and only if every non-sink vertex has either $0$ or $1$ chip on it. By ``unwrapping'' the cycle, we can identify the set of stable configurations with the set of configurations of chips on an $n$-vertex path graph in which each vertex has either $0$ or $1$ chip. We imagine that the vertices of this path are drawn on a horizontal line so that the source is the leftmost vertex. In this set-up, whenever one of the endpoint vertices fires, it sends one of its chips to its neighbor in the path, while another one of its chips disappears. We can naturally identify the set of stable configurations with the set $\{0,1\}^n$ of binary strings of length $n$. 

Starting with a stable configuration of chips, we add $1$ chip to the source vertex and then repeatedly fire vertices until reaching another stable configuration. This procedure defines a function $\chi:\{0,1\}^n\to\{0,1\}^n$. For example, suppose $n=3$, and let us begin with the string $110$. We add $1$ to the first entry to get $210$. The leftmost vertex fires, resulting in the configuration represented by $020$. Now the second vertex fires, resulting in the configuration represented by $101$, which is stable. Thus, $\chi(110)=101$. 

\begin{theorem}\label{Thm:Chip}
For every $n\geq 2$, the degree of $\chi:\{0,1\}^n\to\{0,1\}^n$ is $3/2$. 
\end{theorem}

\begin{proof}
Let $A_i=\{w\in\{0,1\}^n:|\chi^{-1}(w)|=i\}$. A straightforward analysis of how a binary string can have preimages under $\chi$ shows that every string must have at most $2$ preimages. Furthermore, a string has exactly $2$ preimages if and only if it is $1^{n-1}0$ or it is of the form $1^k01x$ for some $k\in\{1,\ldots,n-2\}$ and some $x\in\{0,1\}^{n-k-2}$. It follows that $|A_2|=1+\sum_{k=1}^{n-2}2^{n-k-2}=2^{n-2}$. We also have \[2\cdot 2^{n-2}+1\cdot |A_1|+0\cdot|A_0|=\sum_{w\in\{0,1\}^n}|\chi^{-1}(w)|=2^n,\] so $|A_1|=2^{n-1}$. Therefore, \[\deg(\chi:\{0,1\}^n\to\{0,1\}^n)=\frac{1}{2^n}(|A_0|\cdot 0^2+|A_1|\cdot 1^2+|A_2|\cdot 2^2)=\frac 32. \qedhere\]
\end{proof}

\begin{remark}
The maps $\nib:\{0,1\}^n\to\{0,1\}^n$ and $\chi:\{0,1\}^n\to\{0,1\}^n$ have exactly the same degrees, even though these maps are not conjugate. In other words, there does not exist a bijection $g:\{0,1\}^n\to\{0,1\}^n$ such that $\nib\circ g=g\circ\chi$. Indeed, the map $\nib$ has fixed points while $\chi$ does not. However, these two maps are pseudoconjugate (recall the terminology introduced at the end of the introduction). To see this, note that by Proposition~\ref{Prop:Pseudoconjugate}, it suffices to prove that there exists a bijection $h:\{0,1\}^n\to\{0,1\}^n$ such that $|\nib^{-1}(y)|=|\chi^{-1}(h(y))|$ for all $y\in\{0,1\}^n$. The proofs of Theorems~\ref{Thm:BinaryNibble} and \ref{Thm:Chip} show that for each of the maps $\nib:\{0,1\}^n\to\{0,1\}^n$ and $\chi:\{0,1\}^n\to\{0,1\}^n$, there are $2^{n-1}$ elements of $\{0,1\}^n$ with $1$ preimage, $2^{n-2}$ elements of $\{0,1\}^n$ with $2$ preimages, and no elements of $\{0,1\}^n$ with more than $2$ preimages. This immediately implies that the desired bijection $h$ exists. 
\end{remark}

\subsection{Bulgarian and Carolina Solitaire}\label{Sec:Solitaire}

Bulgarian solitaire is a function $\mathscr B$ that sends partitions of $n$ to partitions $n$. It was introduced around 1980 and has been studied extensively ever since because of its many fascinating dynamical properties (see \cite{Drensky, Griggs, Hopkins, Hopkins2} and the references therein). Given a partition $\lambda=(\lambda_1,\ldots,\lambda_\ell)$ of $n$, we define $\mathscr B(\lambda)$ to be the partition obtained by putting the numbers $\ell,\lambda_1-1,\ldots,\lambda_\ell-1$ in nonincreasing order and deleting any $0$'s. For example, $\mathscr B(8,3,3,1,1)=(7,5,2,2)$.

Carolina solitaire, a variant of Bulgarian solitaire introduced in \cite{Griggs} and studied further in \cite{Hopkins2, Tambellini}, is a map that sends compositions of $n$ to compositions of $n$. Given a composition $c=(c_1,\ldots,c_\ell)$ of $n$, we define $\mathscr C(c)$ to be the composition obtained by deleting all of the $0$'s from the tuple $(\ell,c_1-1,\ldots,c_\ell-1)$. For example, $\mathscr C(3,1,3,7,1,8)=(6,2,2,6,7)$. 

It appears that computing $\deg(\mathscr B:\Part(n)\to\Part(n))$ exactly, or even asymptotically, is quite difficult. By contrast, we will compute $\deg(\mathscr C:\Comp(n)\to\Comp(n))$ exactly. 

\begin{theorem}\label{Thm5}
We have $\liminf\limits_{n\to\infty}\deg(\mathscr B:\Part(n)\to\Part(n))\geq 2$. For every positive integer $n$, we have \[\deg(\mathscr B:\Part(n)\to\Part(n))\leq\left\lfloor\dfrac{1+\sqrt{\frac{8}{3}n+1}}{2}\right\rfloor.\]
\end{theorem}

\begin{proof}
Consider a partition $\lambda=(\lambda_1,\ldots,\lambda_\ell)$ of $n$. It is straightforward to check that $|\mathscr B^{-1}(\lambda)|$ is the number of distinct elements of the list $\lambda_1,\ldots,\lambda_\ell$ that are at least $\ell-1$. The \emph{rank} of a partition $\lambda=(\lambda_1,\ldots,\lambda_\ell)$ is defined to be $\lambda_1-\ell$. Thus, the image $\mathscr B(\Part(n))$ consists of the partitions of $n$ with rank at least $-1$ (these were enumerated in \cite{Hopkins2}). Considering the involution of $\Part(n)$ that sends a partition to its conjugate, we see that for each integer $r$, the number of partitions of $n$ with rank $r$ is equal to the number of partitions of $n$ with rank $-r$. In addition, the main result in \cite{Dousse} shows that the number of partitions of $n$ with rank $-1$ or $0$ is $o(|\Part(n)|)$. It follows that $|\mathscr B(\Part(n))|\sim\dfrac{1}{2}|\Part(n)|$ as $n\to\infty$. The first statement in the theorem now follows from Lemma \ref{Lem1}. 

To prove the desired upper bound, let us fix a positive integer $u$. Let $N$ be the smallest positive integer such that there exists a partition $\lambda=(\lambda_1,\ldots,\lambda_\ell)$ of $N$ with $|\mathscr B^{-1}(\lambda)|\geq u$. We now invoke the description of $|\mathscr B^{-1}(\lambda)|$ mentioned in the previous paragraph. If any of the parts of $\lambda$ are at most $\ell-2$, then we can delete these parts to obtain a new partition $\widetilde\lambda$ of an integer $\widetilde N$ with $|\mathscr B^{-1}(\widetilde\lambda)|\geq u$ and $\widetilde N<N$. This contradicts the minimality of $N$, so we must have $\lambda_\ell\geq \ell-1$. We may now assume that $\ell=u$ and that all of the parts of $\lambda$ are distinct. Indeed, if this were not the case, then we could again delete a part of $\lambda$ to obtain a new partition of an integer smaller than $N$ with at least $u$ preimages under $\mathscr B$. It follows that $\lambda_i\geq \lambda_\ell+\ell-i\geq 2\ell-i-1$ for all $i$. Thus, $N=\lambda_1+\cdots+\lambda_\ell\geq (2\ell-2)+(2\ell-3)+\cdots+(\ell-1)=3\dfrac{\ell(\ell-1)}{2}=3\dfrac{u(u-1)}{2}$. This shows that for every positive integer $n$, $\max\limits_{\lambda\in\Part(n)}|\mathscr B^{-1}(\lambda)|\leq w_n$, where $w_n$ is the largest integer such that $3\dfrac{w_n(w_n-1)}{2}\leq n$. It is straightforward to verify that $w_n=\left\lfloor\dfrac{1+\sqrt{\frac{8}{3}n+1}}{2}\right\rfloor$, so the desired result now follows from Lemma~\ref{Lem1}
\end{proof}

At the moment, we do not see a way to improve the estimates for $\deg(\mathscr B:\Part(n)\to\Part(n))$ in the previous theorem. However, we computed $|B^{-1}(B(\lambda))|$ for $100$ random partitions of $1000$ and also for $100$ random partitions of $100000$. In the first case, the data had a mean of $2.95$ and a standard deviation of $0.22$. In the second case, the mean and standard deviation were $2.85$ and $0.17$. This data hints that the asymptotics of these degrees might be remarkably simple. 

\begin{conjecture}\label{Conj1}
$\lim\limits_{n\to\infty}\deg(\mathscr B:\Part(n)\to\Part(n))=3$.
\end{conjecture}

We now turn our attention to Carolina solitaire. Define a sequence of integers $(\eta_n)_{n\geq 0}$ by the generating function equation \[\sum_{n\geq 0}\eta_nx^n=\frac{1-x}{\sqrt{1 - 4x + 4x^2 - 4x^3 + 4x^4}}.\] This is sequence A217661 in \cite{OEIS}.  
 
\begin{theorem}\label{Thm6}
Preserving the above notation, we have \[\deg(\mathscr C:\Comp(n)\to\Comp(n))=\frac{\eta_n}{2^{n-1}}\] for all $n\geq 1$. Thus, \[\deg(\mathscr C:\Comp(n)\to\Comp(n))\sim\frac{1-\rho}{\sqrt{\pi(1-3\rho+2\rho^2-\rho^3)n}}\left(\frac{1}{2\rho}\right)^n,\] where $\rho\approx 0.33933$ is the smallest positive real root of the polynomial $1-4x+4x^2-4x^3+4x^4$.
\end{theorem}

\begin{proof}
Suppose $c=(c_1,\ldots,c_\ell)$ is a composition of $n$. The preimages of $c$ under $\mathscr C$ are the compositions that have $c_1-(\ell-1)$ parts equal to $1$ and whose other parts are (in order) $c_1+1,\ldots,c_\ell+1$. For example, the preimages of $(4,7,2)$ are $(8,3,1,1)$, $(8,1,3,1)$, $(8,1,1,3)$, $(1,8,3,1)$, $(1,8,1,3)$, $(1,1,8,3)$. In general, the number of preimages of $c$ is $\displaystyle{c_1\choose \ell-1}$. The number of compositions of $n$ with first part $c_1$ and with $\ell$ parts is $\displaystyle{n-c_1-1\choose \ell-2}$. Finally, it is well known that $|\Comp(n)|=2^{n-1}$. Thus, \[\deg(\mathscr C:\Comp(n)\to\Comp(n))=\frac{1}{2^{n-1}}\sum_{c_1=1}^n\sum_{\ell=1}^{n-c_1}{n-c_1-1\choose \ell-2}{c_1\choose \ell-1}^2.\] It is known (see the comments in the OEIS entry A217661 \cite{OEIS}) that \[\eta_n=\sum_{\ell=1}^{n-c_1}\sum_{\ell=1}^{n-c_1}{n-c_1-1\choose \ell-2}{c_1\choose \ell-1}^2,\] so this proves the first statement of the theorem. The asymptotic formula for $\deg(\mathscr C:\Comp(n)\to\Comp(n))$ follows from the asymptotic formula for $\eta_n$, which appears in the OEIS entry A217661. 
\end{proof}

\section{Iterates of a Function}\label{Sec:Iterates} 

Now that we have examined the degrees of several specific discrete dynamical systems, we will shift our focus to a problem with a more extremal-combinatorics flavor. Let us fix an integer $k\geq 2$. In this section, we compare the degree of an arbitrary function $f:X\to X$ with the degree of its iterates $f^k:X\to X$. Here $X$ is a finite set of size $n$. We first observe that \[\dfrac{1}{n^{\gamma-1}}\leq\dfrac{\deg(f^k)}{\deg(f)^\gamma}\leq n\] for every real $\gamma\geq 0$. Indeed, the upper bound follows from the fact that $\deg(f^k)\leq n$ and $\deg(f)\geq 1$. To prove the lower bound, notice that if $x,x'\in X$ satisfy $f(x)=f(x')$, then $f^k(x)=f^k(x')$. This proves that $\deg(f^k)\geq \deg(f)$, so \[\frac{\deg(f^k)}{\deg(f)^\gamma}\geq \frac{1}{\deg(f)^{\gamma-1}}\geq \frac{1}{n^{\gamma-1}}.\] Furthermore, the lower bound is attained by a constant function. The main result we will prove below is that if $\gamma\geq 2-1/2^{k-1}$, then the upper bound can be replaced by $n^{1-1/2^{k-1}}$. In addition, we will construct an example showing that if $\gamma$ and $k$ are fixed, then this upper bound is tight up to a constant factor. Hence, for every integer $k\geq 2$ and real number $\gamma\geq 2-1/2^{k-1}$, we have determined how large the ratio $\dfrac{\deg(f^k)}{\deg(f)^\gamma}$ can be, up to a constant factor. 

\begin{theorem}\label{Thm3}
Let $k\geq 2$ be an integer, and let $\gamma\geq 0$ be a real number. As $n\to\infty$, we have \[\max_{\substack{f:X\to X\\ |X|=n}}\frac{\deg(f^k)}{\deg(f)^\gamma}\geq \frac{1+o(1)}{(k+1)^\gamma}n^{1-1/2^{k-1}}.\] If $\gamma\geq 2-1/2^{k-1}$, then \[\max_{\substack{f:X\to X\\ |X|=n}}\frac{\deg(f^k)}{\deg(f)^{\gamma}}\leq n^{1-1/2^{k-1}}.\]
\end{theorem}

\begin{remark}\label{Rem1}
One special case of Theorem \ref{Thm3} tells us how badly the degree of noninvertibility can fail to satisfy the submultiplicativity inequality $\deg(f\circ g)\leq\deg(f)\deg(g)$ in the specific case in which $f=g$. Indeed, if we set $k=\gamma=2$, then the theorem tells us that \[\frac{1+o(1)}{9}\sqrt n\leq\max_{\substack{f:X\to X\\ |X|=n}}\frac{\deg(f^2)}{\deg(f)^2}\leq\sqrt n.\]  
\end{remark}

The first part of Theorem \ref{Thm3} will follow from the next proposition. 

\begin{proposition}\label{Prop1}
Fix $k\geq 2$. For $n\geq 1$, there exist sets $X_n$ with $|X_n|=n$ and functions $f_n:X_n\to X_n$ such that \[\deg(f_n)=k+1+o(1)\quad\text{and}\quad\deg(f_n^k)=n^{1-1/2^{k-1}}(1+o(1))\quad\text{as }n\to\infty.\]
\end{proposition}

\begin{proof}
Let $T(m_0)$ be the rooted tree that is isomorphic to a path with $m_0$ edges (so each non-leaf vertex has exactly one child). Let $T(m_0,m_1)$ be the rooted tree whose root has $m_0$ subtrees, each of which is isomorphic to $T(m_1)$. In general, let $T(m_0,\ldots,m_p)$ be the rooted tree in which the root has $m_0$ subtrees, each of which is isomorphic to $T(m_1,\ldots,m_p)$. Thus, $T(m_0,\ldots,m_p)$ has $1+m_0+m_0m_1+\cdots+m_0m_1\cdots m_p$ vertices. Consider a sufficiently large integer $b$, and let $\mathcal T_b=T(b_0,\ldots,b_k)$, where $b_i=\left\lfloor b^{1/2^i}\right\rfloor$ for $0\leq i\leq k-1$ and $b_k=b_{k-1}=\left\lfloor b^{1/2^{k-1}}\right\rfloor$. Let $X^{(b)}$ be the set of vertices in $\mathcal T_b$. We obtain a function $F_b:X^{(b)}\to X^{(b)}$ by sending the root of $\mathcal T_b$ to itself and sending every non-root vertex in $\mathcal T_b$ to its parent (see Figure \ref{Fig1}). Let \[n_b=|X^{(b)}|=1+b_0+b_0b_1+\cdots +b_0b_1\cdots b_k,\] and note that $n_b=b^2(1+o(1))$ (here and in what follows, the asymptotic notation is taken as $b\to\infty$). 
Our first goal will be to estimate the degrees of $F_b$ and $F_b^k$. We will then prove that for every sufficiently large $n$, we can find an $n$-element set $X_n$, a function $f_n:X_n\to X_n$, and a suitable function $F_b$ such that $\deg(F_b)$ approximates $\deg(f_n)$ and $\deg(F_b^k)$ approximates $\deg(f_n^k)$. 

\begin{figure}[ht]
\begin{center}
\includegraphics[height=4.1cm]{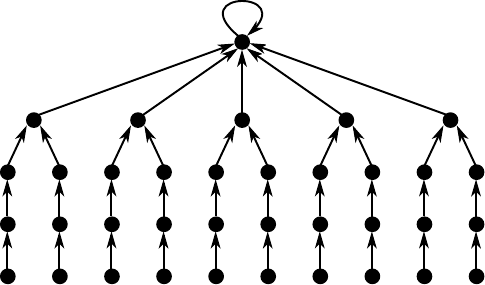}
\caption{The map $F_5$ when $k=2$. In this case, $b=b_0=5$ and $b_1=b_2=2$.}
\label{Fig1}
\end{center}  
\end{figure}

We define the \emph{depth} of a vertex $v$ in $\mathcal T_b$ to be the smallest nonnegative integer $t$ such that $F_b^t(v)$ is the root of $\mathcal T_b$. The root of $\mathcal T_b$ is the only vertex of depth $0$; it has $b_0$ preimages under $F_b$ and $1+\sum_{i=0}^{k-1}b_0b_1\cdots b_i$ preimages under $F_b^k$. For $1\leq t\leq k-1$, there are $b_0b_1\cdots b_{t-1}$ vertices of depth $t$, each of which has $b_t$ preimages under $F_b$ and $b_tb_{t+1}\cdots b_{k-1}$ preimages under $F_b^k$. There are $b_0b_1\cdots b_{k-1}$ vertices of depth $k$, each of which has $1$ preimage under $F_b$ and $1$ preimage under $F_b^k$. For $k+1\leq t\leq k+b_k-1$, there are $b_0b_1\cdots b_{k-1}$ vertices of depth $t$, each of which has $1$ preimage under $F_b$ and has $0$ preimages under $F_b^k$. Finally, there are $b_0b_1\cdots b_{k-1}$ vertices of depth $k+b_k$ (the leaves), each of which has $0$ preimages under $F_b$ and under $F_b^k$. Combining this information, we find that \[\deg(F_b)=\frac{1}{n_b}\left[b_0^2+b_0b_1^2+b_0b_1b_2^2+\cdots+b_0b_1\cdots b_{k-1}b_k^2+\sum_{t=k}^{b_k+k-1}b_0b_1\cdots b_{k-1}\cdot 1^2\right]\] \[=\frac{1}{n_b}[b_0^2+b_0b_1^2+b_0b_1b_2^2+\cdots+b_0b_1\cdots b_{k-1}b_k^2+b_0b_1\cdots b_{k-1}b_k]\]  
\begin{equation}\label{Eq4}
=\frac{1}{b^2(1+o(1))}(k+1)b^2(1+o(1))=k+1+o(1)
\end{equation} and 

\[\deg(F_b^k)\!=\!\frac{1}{n_b}\hspace{-2.53pt}\left[\left(1+\sum_{i=0}^{k-1}b_0b_1\cdots b_i\right)^2+\sum_{t=1}^{k-1}b_0b_1\cdots b_{t-1}(b_tb_{t+1}\cdots b_{k-1})^2+b_0b_1\cdots b_{k-1}\cdot 1^2\right]\]
\begin{equation}\label{Eq2}
=\frac{1}{n_b}\left(b_0b_1\cdots b_{k-1}\right)^2(1+o(1))=\frac{1}{n_b}(b^{2-1/2^{k-1}})^2(1+o(1))=n_b^{1-1/2^{k-1}}(1+o(1)).
\end{equation}

Now let $n$ be a sufficiently large integer, and let $b$ be the unique integer such that $n_b\leq n<n_{b+1}$. Let $X_n=X^{(b)}\cup Y_n$, where $Y_n$ is a set of size $n-n_b$ that is disjoint from $X^{(b)}$. Define $f_n:X_n\to X_n$ by $f_n(x)=F_b(x)$ for $x\in X^{(b)}$ and $f_n(y)=y$ for $y\in Y_n$. It is straightforward to check that \[n_b\deg(F_b)\leq n\deg(f_n)\leq n_{b+1}\deg(F_{b+1}),\] so it follows from \eqref{Eq4} that $\dfrac{n_b}{n}(k+1+o(1))\leq\deg(f_n)\leq\dfrac{n_{b+1}}{n}(k+1+o(1))$. Since $n_b=b^2(1+o(1))$ and $n_{b+1}=b^2(1+o(1))$, we have $\dfrac{n_b}{n}=1+o(1)$ and $\dfrac{n_{b+1}}{n}=1+o(1)$. Consequently, \[\deg(f_n)=k+1+o(1).\] A similar argument shows that \[\deg(f_n^k)=n^{1-1/2^{k-1}}(1+o(1)). \qedhere\]
\end{proof}

In an earlier draft of this article, we proved that for $f:X\to X$ with $|X|=n$, the inequality $\dfrac{\deg(f^k)}{\deg(f)^\gamma}\leq\beta_k^{2-1/2^{k-1}}n^{1-1/2^{k-1}}$ holds, where $\beta_k$ depends only on $k$ and satisfies $\beta_k\sim k/4$ as $k\to\infty$. We are very grateful to Yan Sheng Ang, who showed us the proof of the following theorem. This theorem is more general than our original result, and it immediately implies an improved upper bound for $\dfrac{\deg(f^k)}{\deg(f)^\gamma}$ (without the $\beta_k^{2-1/2^{k-1}}$ factor). We have seen that the submultiplicativity inequality $\deg(f\circ g)\leq\deg(f)\deg(g)$ does not hold in general; the next theorem can be viewed as a corrected version of this inequality. 

\begin{theorem}\label{Thm7}
Let $X$ be an $n$-element set. For all functions $f,g:X\to X$, we have the inequality \[\deg(f\circ g)\leq\sqrt{n}\sqrt{\deg(f)}\deg(g).\] Furthermore, equality holds if and only if $f$ is a constant function and $g$ is a bijection.  
\end{theorem}

\begin{proof}
Let $f(g(X))=\{x_1,\ldots,x_r\}$. For $1\leq i\leq r$, let $f^{-1}(x_i)=\{y_{i1},\ldots,y_{ia_i}\}$. For $1\leq i\leq r$ and $1\leq j\leq a_i$, let $b_{ij}=|g^{-1}(y_{ij})|$. We have 
\begin{equation}\label{Eq3}
n\deg(f\circ g)=\sum_{i=1}^r|g^{-1}(f^{-1}(x_i))|^2=\sum_{i=1}^r\left(\sum_{j=1}^{a_i}1\cdot b_{ij}\right)^2\leq\sum_{i=1}^ra_i\sum_{j=1}^{a_i}b_{ij}^2,
\end{equation} where we have applied the Cauchy-Schwarz inequality in the last step. Since 
\begin{equation}\label{Eq6}
n\deg(f)\geq\sum_{i=1}^r|f^{-1}(x_i)|^2=\sum_{i=1}^ra_i^2,
\end{equation} we have 
\begin{equation}\label{Eq7}
n\deg(f)\geq a_i^2\quad\text{for all }1\leq i\leq r.
\end{equation}
Consequently, \[n\deg(f\circ g)\leq\sum_{i=1}^ra_i\sum_{j=1}^{a_i}b_{ij}^2\leq \sqrt{n}\sqrt{\deg(f)}\sum_{i=1}^r\sum_{j=1}^{a_i}b_{ij}^2.\]
Because $g(X)\subseteq f^{-1}(f(g(X)))=\{y_{ij}:1\leq i\leq r, 1\leq j\leq a_i\}$, we have $\displaystyle n\deg(g)=\sum_{i=1}^r\sum_{j=1}^{a_i}b_{ij}^2$. This shows that \[n\deg(f\circ g)\leq \sqrt{n}\sqrt{\deg(f)}n\deg(g),\] and the desired inequality follows. 

If $f$ is a constant function and $g$ is a bijection, then the inequality in the theorem is an equality because $\deg(f)=\deg(f\circ g)=n$ and $\deg(g)=1$. To prove the converse, suppose the inequality is an equality. This means the inequality in \eqref{Eq6} must be an equality, so $f(X)=\{x_1,\ldots,x_r\}$. The inequality in \eqref{Eq7} must also be an equality for all $1\leq i\leq r$. Referring to \eqref{Eq6}, we see that this forces $r=1$. Thus, $f(X)=\{x_1\}$, so $f$ is a constant function. It follows that $a_1=|f^{-1}(x_1)|=n$, so $\{y_{11},\ldots,y_{1n}\}=X$. Finally, the inequality in \eqref{Eq3} is an equality. This can only happen if the numbers $b_{1j}$ are equal for all $1\leq j\leq n$. This, in turn, forces $g$ to be a bijection.  
\end{proof}

\begin{proof}[Proof of Theorem \ref{Thm3}]
As mentioned above, the first statement of the theorem follows from Proposition \ref{Prop1}. To prove the second statement, note that it suffices to prove the case in which $\gamma=2-1/2^{k-1}$. Choose $f:X\to X$. We want to prove that \[\deg(f^k)\leq\deg(f)^{2-1/2^{k-1}}n^{1-1/2^{k-1}}\quad\text{for all }k\geq 1.\] This is certainly true if $k=1$, so we may assume $k\geq 2$ and induct on $k$. By Theorem \ref{Thm7} and our induction hypothesis, we have \[\deg(f^k)=\deg(f^{k-1}\circ f)\leq\sqrt{n}\sqrt{\deg(f^{k-1})}\deg(f)\leq\sqrt{n}\sqrt{\deg(f)^{2-1/2^{k-2}}n^{1-1/2^{k-2}}}\deg(f)\] \[=\sqrt{n}\deg(f)^{1-1/2^{k-1}}n^{1/2-1/2^{k-1}}\deg(f)=\deg(f)^{2-1/2^{k-1}}n^{1-1/2^{k-1}}. \qedhere\]   
\end{proof}

\section{Future Directions}\label{Sec:Conclusion}  

In this section, we list some possible directions for extending the investigation of degrees of noninvertibility. Of course, a natural place to start would be to consider other specific families of combinatorially interesting discrete dynamical systems. Even restricting our attention to the specific families considered in Section~\ref{Sec:Specific}, there are several problems that remain open. 

To state our first problem, we require the maps $t_i:S_n\to S_n$ from Section~\ref{Sec:Bubble}. The monoid $H_0(S_n)=\langle t_1,\ldots,t_{n-1}\rangle$ generated by the operators $t_1,\ldots,t_{n-1}$ is known as the $0$-Hecke monoid\footnote{Knuth calls the elements of $H_0(S_n)$ \emph{primitive sorting maps} \cite{Knuth2}.} of $S_n$ \cite{Hivert}. Each element $T\in H_0(S_n)$ is a function from $S_n$ to $S_n$ given by $T=t_{i_r}\circ \cdots\circ t_{i_1}$ for some $i_1,\ldots,i_r\in[n-1]$ (we allow the list $i_1,\ldots,i_r$ to contain repeats). We say $T$ is \emph{eventually constant} if there exists a positive integer $k$ such that $T^k$ is the constant function that sends every permutation in $S_n$ to the identity permutation $123\cdots n$. One can show that $T$ is eventually constant if and only if every element of $[n-1]$ appears at least once in the list $i_1,\ldots,i_r$. Notice that the bubble sort map ${\bf B}:S_n\to S_n$ is an eventually constant element of $H_0(S_n)$. It would be interesting to study the degrees of other eventually constant elements of the $0$-Hecke monoid of $S_n$. 

For example, if $r_n^{\text{odd}}$ (respectively, $r_n^{\text{even}}$) denotes the largest odd (respectively, even) element of $[n-1]$, then we define $T_{\text{odd}}=t_{r_n^{\text{odd}}}\circ\cdots\circ t_5\circ t_3\circ t_1$ and $T_{\text{even}}=t_{r_n^{\text{even}}}\circ\cdots\circ t_6\circ t_4\circ t_2$. We then let \[T_{\alt}=T_{\text{even}}\circ T_{\text{odd}}\quad\text{and}\quad T_{\tla}=T_{\text{odd}}\circ T_{\text{even}}.\] Note that the map $T_{\alt}$ is known as the odd-even sort \cite{Knuth2, Lakshmivarahan}. The following conjecture states that among all eventually constant elements of $H_0(S_n)$, bubble sort is the closest to being invertible, while $T_{\tla}$ is the farthest. 

\begin{conjecture}\label{Conj2}
If $T\in H_0(S_n)$ is eventually constant, then \[\deg({\bf B}:S_n\to S_n)\leq \deg(T:S_n\to S_n)\leq\deg(T_{\tla}:S_n\to S_n).\] 
\end{conjecture}

It would be interesting to have a proof of either (if not both) of the inequalities in Conjecture~\ref{Conj2}. We know by Theorem \ref{Thm2} that $\deg({\bf B}:S_n\to S_n)=\dfrac{n(n+1)}{6}$. However, we do not know $\deg(T_{\tla}:S_n\to S_n)$. What is fascinating is that while $\deg({\bf B}:S_n\to S_n)$ grows quadratically in $n$, $\deg(T_{\alt}:S_n\to S_n)$ and $\deg(T_{\tla}:S_n\to S_n)$ grow exponentially. Indeed, one can show that $|T_{\alt}(S_n)|$ and $|T_{\tla}(S_n)|$ are both equal to the number of up-down permutations of length $n$. It then follows from Lemma \ref{Lem1} and the known asymptotic formula for the number of up-down permutations of length $n$ (see sequence A000111 in \cite{OEIS}) that $\deg(T_{\alt}:S_n\to S_n)$ and $\deg(T_{\tla}:S_n\to S_n)$ are both at least $\dfrac{n!}{|T_{\tla}(S_n)|}\sim \dfrac{\pi}{4}\left(\dfrac{\pi}{2}\right)^n$.   

\begin{problem}
Find improved asymptotic estimates (or even exact formulas!) for \[\deg(T_{\alt}:S_n\to S_n)\quad\text{and}\quad \deg(T_{\tla}:S_n\to S_n).\] 
\end{problem}

Let us remark that $\deg(T_{\alt}:S_n\to S_n)=\deg(T_{\tla}:S_n\to S_n)$ when $n$ is odd (this equality can fail when $n$ is even). In fact, for $n$ odd, $T_{\alt}$ and $T_{\tla}$ are dynamically equivalent. To see this, define the \emph{reverse complement} of a permutation $\pi=\pi_1\cdots\pi_n$ to be the permutation $\rc(\pi)=(n+1-\pi_n)(n+1-\pi_{n-1})\cdots(n+1-\pi_1)$. One can show that $T_{\alt}(\rc(\pi))=\rc(T_{\tla}(\pi))$ whenever $\pi\in S_n$ and $n$ is odd. 

We now turn our attention to the stack-sorting map $s$. 

\begin{problem}
Improve the estimates for $\lim\limits_{n\to \infty}\deg(s:S_n\to S_n)^{1/n}$ from Theorem \ref{Thm2}.
\end{problem}

Using the ``decomposition lemma" described in \cite{DefantCounting}, we have computed $|s^{-1}(s(\pi))|^{1/n}$ for several random permutations in $S_n$ in order to gauge the size $\lim\limits_{n\to \infty}\deg(s:S_n\to S_n)^{1/n}$. We first computed this quantity for $1000$ random permutations with $n=100$; the mean and standard deviation were $1.69$ and $0.09$. We then tried $100$ random permutations with $n=300$; the mean and standard deviation were $1.70$ and $0.06$. 

\begin{conjecture}
The value of the limit $\lim\limits_{n\to \infty}\deg(s:S_n\to S_n)^{1/n}$ lies in the interval $(1.68,1.73)$. 
\end{conjecture}

Let us also recall Conjecture~\ref{Conj1} from Section~\ref{Sec:Solitaire}, which states that \[\lim\limits_{n\to\infty}\deg(\mathscr B:\Part(n)\to\Part(n))=3.\]

Theorem \ref{Thm3} pins down, up to a constant factor, the asymptotics of $\max\limits_{\substack{f:X\to X\\ |X|=n}}\dfrac{\deg(f^k)}{\deg(f)^\gamma}$ as $n\to\infty$. It would be nice to improve the constants appearing in this theorem. Of particular interest is when $\gamma=2-1/2^{k-1}$ since this is the minimum $\gamma$ for which the second statement of the theorem applies. Thus, we have the following more specific problem. 

\begin{question}\label{Quest1}
For fixed $k\geq 2$, does the limit \[\lim_{n\to\infty}\max_{\substack{f:X\to X\\ |X|=n}}\frac{\deg(f^k)}{\deg(f)^{2-1/2^{k-1}}}\frac{1}{n^{1-1/2^{k-1}}}\] exist? If so, what is its value? 
\end{question}

Even answering Question \ref{Quest1} for $k=2$ would be quite interesting. So far, we know from Theorem~\ref{Thm3} that, if this limit exists, its value is between $3^{-3/2}\approx 0.19245$ and $1$. 

\section{Acknowledgements}\label{sec:acknowledgements}
We thank Vic Reiner for suggesting that we examine the Bulgarian solitaire map and thank Yan Sheng Ang for showing us the statement and proof of Theorem~\ref{Thm7}.

\end{document}